\pdfoutput=1
\RequirePackage{ifpdf}
\ifpdf 
\documentclass[pdftex]{sigma}
\else
\documentclass{sigma}
\fi

\numberwithin{equation}{section}

\newtheorem{Theorem}{Theorem}[section]
\newtheorem{Corollary}[Theorem]{Corollary}
\newtheorem{Lemma}[Theorem]{Lemma}
\newtheorem{Proposition}[Theorem]{Proposition}
\newtheorem{prop}[Theorem]{Result}
 { \theoremstyle{definition}
\newtheorem{Definition}[Theorem]{Definition}

\newtheorem{Remark}[Theorem]{Remark} }

\newcommand{\Cc}{\mathcal{C}}

\newcommand{\hmod}{{_H\mathcal{M}}}

\newcommand{\vect}{\text{Vec}}

\newcommand{\ot}{\otimes}

\newcommand{\zp}{\mathbb{Z}/p}
\begin{document}

\newcommand{\arXivNumber}{2010.02768}

\renewcommand{\PaperNumber}{026}

\FirstPageHeading

\ShortArticleName{Mixed vs Stable Anti-Yetter--Drinfeld Contramodules}

\ArticleName{Mixed vs Stable Anti-Yetter--Drinfeld Contramodules}

\Author{Ilya SHAPIRO}

\AuthorNameForHeading{I.~Shapiro}

\Address{Department of Mathematics and Statistics, University of Windsor,\\ 401 Sunset Avenue, Windsor, Ontario N9B 3P4, Canada}
\Email{\href{mailto:ishapiro@uwindsor.ca}{ishapiro@uwindsor.ca}}
\URLaddress{\url{http://http://web2.uwindsor.ca/math/ishapiro/}}

\ArticleDates{Received November 09, 2020, in final form March 04, 2021; Published online March 17, 2021}

\Abstract{We examine the cyclic homology of the monoidal category of modules over a finite dimensional Hopf algebra, motivated by the need to demonstrate that there is a difference between the recently introduced mixed anti-Yetter--Drinfeld contramodules and the usual stable anti-Yetter--Drinfeld contramodules. Namely, we show that Sweedler's Hopf algebra provides an example where mixed complexes in the category of stable anti-Yetter--Drinfeld contramodules (previously studied) are not equivalent, as differential graded categories to~the category of mixed anti-Yetter--Drinfeld contramodules (recently introduced).}


\Classification{16E35; 16T05; 18G90; 19D55}

\section{Introduction}

Cyclic (co)homology was introduced independently by Boris Tsygan and Alain Connes in the 1980s. It has since been generalized, applied to many fields, and now reaches into many different settings. Our investigations in this paper focus on the equivariant flavour that began with Connes--Moscovici~\cite{conmos} and was generalized into Hopf-cyclic cohomology by Hajac--Khalkhali--Rangipour--Sommerh\"{a}user~\cite{HKRS1, HKRS2} and Jara--Stefan~\cite{JS} (independently). Roughly speaking, the original theory defines cohomology groups for an associative algebra that play the role of the de~Rham cohomology in the noncommutative setting. The equivariant version considers an~alge\-bra with an action of a Hopf alge\-bra. It turns out that just as in the de~Rham cohomology, one has coefficients in the Hopf setting; it is an interesting fact that, unlike the de~Rham setting, Hopf-cyclic cohomology requires coefficients, i.e., there are no canonical trivial coefficients. These coefficients are known as stable anti-Yetter--Drinfeld modules, due to their similarity to the usual Yetter--Drinfeld modules. It turns out that the more natural, from a conceptual point of view, version of coefficients are stable anti-Yetter--Drinfeld contramodules~\cite{contra}. It is the desire to~under\-stand the coefficients themselves that motivated a series of papers by the author of the present one. This paper is a natural next step.

This paper is a descendant of~\cite{chern}, where it is shown that the classic stable anti-Yetter--Drinfeld contramodules are simply objects in the naive cyclic homology category of $\hmod$, the monoidal category of modules over the Hopf algebra $H$. It is furthermore conjectured there, that the new coefficients introduced (mixed anti-Yetter--Drinfeld contramodules) are obtained via the true cyclic homology category; this makes exact the analogy between the de~Rham coefficients in the geometric setting and the Hopf-cyclic coefficients. Namely, while the latter are obtained from the cyclic homology of $\hmod$, the former are shown in~\cite{it} to arise from the cyclic homology of quasi-coherent sheaves on the space $X$. More precisely, in~\cite{chern}, a category of mixed anti-Yetter--Drinfeld contramodules is defined by analogy with the derived algebraic geometry case of~\cite{it}. This new generalization is conceptual, and furthermore allows the expression of~the Hopf-cyclic cohomology of an algebra $A$ with coefficients in $M$ as an ${\rm Ext}$ (in this category) between $\mathop{\rm ch}(A)$, the Chern character object associated to $A$, and $M$ itself. Even if one takes $M$ to be a~stable anti-Yetter--Drinfeld contramodule, the object $\mathop{\rm ch}(A)$ is truly a mixed anti-Yetter--Drinfeld contramodule. It is conjectured that mixed anti-Yetter--Drinfeld contramodules are the objects in the cyclic homology category of $\hmod$.

The comparison in~\cite{chern} between anti-Yetter--Drinfeld contramodules and the cyclic homology category of $\hmod$ involves a monad on $\hmod$ with a central element $\sigma$. When we talk about the $S^1$-action we mean the action of this central element on the category of modules over the monad. It is this description that allows us here to reduce the investigations into the differences between the previously studied and the new Hopf-cyclic cohomology to the analysis of categories of modules over two differential graded algebras (DGAs). Namely, in the notation of the paper, we have an algebra $\widehat{D}(H)$ whose modules are the anti-Yetter--Drinfeld contramodules, we have a DGA $\widehat{D}(H)[\theta]$ with $d\theta=\sigma-1$ that yields the new mixed anti-Yetter--Drinfeld contramodules, and we have a DGA $\widehat{D}(H)/(\sigma-1)[\theta]$ with $d\theta=0$ that yields the previously studied setting, i.e., the mixed complexes in stable anti-Yetter--Drinfeld contramodules. Thus, it suffices for our purposes to compare the DG categories of modules over these two DGAs. We~concentrate on~finite dimensional Hopf algebras $H$. We~show that if the square of the antipode is trivial, i.e., $S^2={\rm Id}$ then the DG categories coincide (Proposition~\ref{sscase}):
\begin{prop}
Let $H$ be a finite dimensional Hopf algebra such that the square of the antipode is equal to the identity, i.e., $S^2={\rm Id}$. Then the categories of mixed complexes in stable $aYD$-contramodules and mixed $aYD$-contramodules are $DG$-equivalent.
\end{prop} On the other hand, if we consider Sweedler's Hopf algebra $T_2(-1)$ (the smallest case of $S^2\neq {\rm Id}$) then they do not (Proposition~\ref{notsscase}):
\begin{prop}
Let $H=T_2(-1)$, then the mixed complexes in the category of stable anti-Yetter--Drinfeld contramodules are not DG-equivalent to the category of mixed anti-Yetter--Drinfeld contramodules.
\end{prop}

\noindent \textbf{Conventions.} All algebras $A$ in monoidal categories are assumed to be unital associative. Our~$H$ is a Hopf algebra over some fixed algebraically closed field $k$, of characteristic $0$, and $\vect$ denotes the category of $k$-vector spaces. For the purposes of this paper we are only interested in finite dimensional Hopf algebras. We~use the following version of Sweedler's notation: For~$h\in H$ we denote the coproduct $\Delta(h)\in H\ot H$ by $h^1\ot h^2$. The letter $S$ denotes the antipode of~$H$. The number $p$ is prime. Finally, DG stands for differential graded.

\section{Twisted Drinfeld double}
Let $H$ be a Hopf algebra. From~\cite{chern} we see that the study of the Hochschild and cyclic homologies of $\hmod$, the monoidal category of $H$-modules, reduces to the study of modules over a certain monad on $\hmod$. Recall that the consideration of Hochschild and cyclic homologies of monoidal categories is motivated by their recently discovered role~\cite{chern} in the understanding of Hopf-cyclic theory coefficients.

Briefly, we have the monad (see~\cite{chern} for more details):
\begin{gather}\label{themonad}
\mathop{\rm Hom}\nolimits_k(H,-)\colon \ \hmod\to\hmod
\end{gather}
with the $H$-module structure on $\mathop{\rm Hom}\nolimits_k(H,V)$:
\begin{gather}
x\cdot\varphi=x^2\varphi\big(S\big(x^3\big)(-)x^1\big),
\end{gather}
for $x\in H$ and $\varphi\in \mathop{\rm Hom}\nolimits_k(H,V)$. The unit $1_V\colon {\rm Id}(V)\to \mathop{\rm Hom}\nolimits_k(H,V)$ is
\begin{gather}
1_V(v)(h)=\epsilon(h)v
\end{gather}
and a crucial central element $\sigma_V\colon {\rm Id}(V)\to \mathop{\rm Hom}\nolimits_k(H,V)$ is
\begin{gather}
\sigma_V(v)(h)=hv.
\end{gather}

The anti-Yetter--Drinfeld contramodules then coincide with modules over this monad (shown in~\cite{chern}), while the stable ones consist of those for which the action of $\sigma$ agrees with that of $1$, and the mixed ones introduced in~\cite{chern} are the homotopic version of this on the nose requirement.

Recall the mixed complexes of~\cite{kassel}. These are complexes of vector spaces $(V^\bullet,d)$ with a~homo\-topy $h$ such that $dh+hd=0$. We~can replace vector spaces with $R$-modules for some ring~$R$. Note that as observed in~\cite{kassel}, the DG-category of mixed complexes in $R$-modules is isomorphic to the DG-category of DG-modules over $R[\theta]$, where $\theta$ is a freely and centrally adjoined degree~$-1$~graded commutative variable (naturally $R$ itself is placed in degree $0$, so that $R[\theta]=R\to R$ as a~comp\-lex) and $d=0$ on $R[\theta]$. The action of $\theta$ gives the homotopy $h$.

We can generalize the considerations of~\cite{kassel} so as to apply to our particular situation. Namely, let $z\in Z(R)$, i.e., $zr=rz$ for all $r\in R$. Define a DG-algebra $R[\theta]$ by placing $R$ in degree $0$ and $\theta$ in degree $-1$. Let $\theta$ commute with $R$ and itself, so in particular $\theta^2=0$. So far it is as above. Now define the differential to be $0$ on $R$ and $d\theta=z$. This is well defined and unique by the Leibniz rule. We~observe that the category of DG-modules over $R[\theta]$ consists of complexes of $R$-modules equipped with a homotopy $h$ such that $dh+hd=z$.

Now recall from~\cite{chern}:

\begin{Definition}
We say that $(M^\bullet,d,h)$ is a mixed anti-Yetter--Drinfeld contramodule if $(M^\bullet,d)$ is a complex of contramodules, i.e., modules over the monad \eqref{themonad}, and $h$ is a homotopy annihilating $\sigma-1$. More precisely,
\begin{gather}
dh+hd=\sigma-1.
\end{gather}
\end{Definition}

In this section we will define an explicit DG-algebra that will yield the mixed anti-Yetter--Drinfeld (aYD) contramodules (for $H$ finite dimensional) as its DG-modules. The construction of the twisted convolution algebra below is analogous to the classical Drinfeld double $D(H)$ and its anti-version $D_a(H)$~\cite{HKRS2} (we review these in Appendix~\ref{appx}, where we expand upon this comparison).

\begin{Definition}\label{dhhat}
Let $H$ be a Hopf algebra with an invertible antipode $S$, define a twisted double~$\widehat{D}(H)$ as follows. The multiplication on $\widehat{D}(H):={\rm End}(H)$ is
\begin{gather}\label{mult}
(f\star g)(h)= f\big(h^1\big)^2g\big(S\big(f\big(h^1\big)^3\big)h^2f\big(h^1\big)^1\big),
\end{gather}
thus the multiplicative identity, which we denote by $1$ is $\epsilon(-)1$, and the central element $\sigma(h)=h$ is invertible with inverse $S^{-1}$.
\end{Definition}

Definition~\ref{dhhat} is extracted from the monad~\eqref{themonad} with the sole purpose consisting of making the following lemma a tautology.

\begin{Lemma}\label{bas}
Let $H$ be a finite dimensional Hopf algebra.
\begin{itemize}\itemsep=0pt
\item The category of anti-Yetter--Drinfeld contramodules over $H$ is isomorphic to $\widehat{D}(H)$-mo\-dules.

\item The category of stable anti-Yetter--Drinfeld contramodules is isomorphic to modules over
\begin{gather}
A:=\widehat{D}(H)/(\sigma-1).
\end{gather}

\item The DG-category of mixed anti-Yetter--Drinfeld contramodules is isomorphic to DG-mo\-dules over the DG algebra
\begin{gather}
B:=\widehat{D}(H)[\theta],
\end{gather}
where $\theta$ is a freely adjoined degree $-1$ graded commutative variable and $d\theta=\sigma-1$, with $d|_{\widehat{D}(H)}=0$.
\end{itemize}
\end{Lemma}
\begin{proof}
In the finite dimensional case, as vector spaces, $\widehat{D}(H)={\rm End}(H)\simeq H^*\ot H$. Furthermore, as an algebra, $\widehat{D}(H)$ is the quotient of the free product algebra, generated by $H^*$ and $H$, by the relation:
\begin{gather}
\label{mult1}
h \chi= \chi\big(S\big(h^3\big)(-)h^1\big)h^2,
\end{gather}
where $h\in H$, $\chi\in H^*$. Thus, modules over the algebra are both $H$-modules and $H$-cont\-ra\-mo\-du\-les (same as $H^*$-modules for $H$ finite dimensional). The two actions satisfy the requisite compatibility condition for contramodules, as specified in~\cite{contra}, and ensured by \eqref{mult1}. Clearly, modules over $A$ consist of the full subcategory of objects on which $\sigma$ acts by identity, these are exactly the stable contramodules. Finally, a DG-module over $B$ is just a complex of $\widehat{D}(H)$-modules with a homotopy given by the action of $\theta$. The condition $d\theta=\sigma-1$ ensures that $dh+hd=\sigma-1$ on~$M^\bullet$.
\end{proof}

Our main goal is to compare the category of mixed
aYD contramodules to the category of~mixed complexes of stable aYD
contramodules. By the preceding lemma this means determining when,
and more interestingly when not, the category of DG-modules over $B$ is DG-equivalent to $A[\theta]$-modules (with $\theta$ of degree $-1$ and $d=0$). The study of Hopf-cyclic cohomology has thus far only concerned itself with the latter.

The following simple lemma takes care of a lot of cases.

\begin{Lemma}\label{diag}
Let $H$ be a finite dimensional Hopf algebra and suppose that the action of $\sigma-1$ on~$\widehat{D}(H)$ is diagonalizable. Then the categories of mixed complexes in stable $aYD$-contramodules and mixed $aYD$-contramodules are $DG$-equivalent.
\end{Lemma}

\begin{proof}
Since the action of the central element $\sigma-1$ on $\widehat{D}(H)$ is diagonalizable we may decom\-pose $\widehat{D}(H)$ as a product of algebras $\widehat{D}(H)_0\oplus \widehat{D}(H)_+$ with $\widehat{D}(H)_0$ being the $0$-eigenspace and~$\widehat{D}(H)_+$ all the other eigenspaces. We~have an inclusion of DGAs: $\widehat{D}(H)_0[\theta]\to B$ that induces an isomorphism on cohomology. Namely, as complexes $B$ is $\widehat{D}(H)\stackrel{\sigma-1}{\to}\widehat{D}(H)$ whereas $\widehat{D}(H)_0[\theta]$ is $\widehat{D}(H)_0\stackrel{0}{\to}\widehat{D}(H)_0$ and $\sigma-1$ is invertible on $\widehat{D}(H)_+$. Note that $\widehat{D}(H)_0\simeq A$ and we are done.
\end{proof}

\begin{Proposition}\label{sscase}
Let $H$ be a finite dimensional Hopf algebra such that the square of the antipode is identity, i.e., $S^2={\rm Id}$. Then the categories of mixed complexes in stable $aYD$-contramodules and mixed $aYD$-contramodules are $DG$-equivalent.
\end{Proposition}

\begin{proof}We need characteristic $0$ here. Since $S^2={\rm Id}$, so $H$ is semi-simple~\cite{sqid2, sqid}, so $D(H)$ (its Drinfeld double) is semi-simple~\cite{dhss}. By Lemma~\ref{uhu}, we know that it follows that $D_a(H)$ is semi-simple and thus by~\cite{s2} so is $\widehat{D}(H)$. Thus, by Schur's lemma, the action of the central element $\sigma-1$ is diagonalizable and we are done by Lemma~\ref{diag}.
\end{proof}

In light of the above we need to consider an example of $H$ with $S^2\neq {\rm Id}$. It turns out that the smallest, dimension wise, such example suffices.

\section{Taft Hopf algebras}\label{sec:blah}
We fix a prime $p$ and a primitive $p$th root of unity $\xi\in k$ in the following. The Taft Hopf algebra $T_p(\xi)$~\cite{taft} is generated as a $k$-algebra by $g$ and $x$ with the relations \begin{gather}
g^p=1,\qquad
x^p=0,
\\
gx=\xi xg.
\end{gather}
It is sometimes called the quantum ${\mathfrak{sl}}_2$ Borel algebra. It is $p^2$ dimensional over $k$. Furthermore, the coalgebra structure is
\begin{gather}
\Delta(g)=g\ot g,\qquad \Delta(x)=x\ot 1+g\ot x
\end{gather}
with $\epsilon(g)=1$, $\epsilon(x)=0$, and thus $S(g)=g^{-1}$, while $S(x)=-g^{-1}x$. Note that
\begin{gather}
S^2(x)=\xi^{-1}x\neq x,
\end{gather}
making $T_2(-1)$ the smallest Hopf algebra with $S^2\neq {\rm Id}$. The Taft algebra $T_2(-1)$ is somewhat different from the other $T_p(\xi)$ and has its own name: Sweedler’s Hopf algebra.

\subsection{The identification with the dual}
The Taft algebra is isomorphic to its dual. We~need some explicit formulas establishing the isomorphism $T_p(\xi)\simeq T_p(\xi)^*$ and its inverse. Suppose that $\omega$ is a $p$th root of unity, let
\begin{gather}
(n)_\omega=1+\cdots +\omega^{n-1}
\end{gather}
and
\begin{gather}
(n)_\omega !=(n)_\omega\cdots(1)_\omega.
\end{gather}

The verification of the following is left to the reader; key details can be found in~\cite{selfdual}. The lemma itself can be obtained from~\cite{nen}.
\begin{Lemma}\label{nenlem}
 As Hopf algebras
 \begin{gather}
 T_p(\xi)^*\simeq T_p(\xi).
 \end{gather}
\end{Lemma}
\begin{proof}
Consider a basis of $T_p(\xi)$: $\big\{g^ix^j\big\}_{ i,j=0}^{p-1}$ so that $\big\{\big(g^ix^j\big)^*\big\}$ denotes the dual basis of $T_p(\xi)^*$. Then the isomorphism of Hopf algebras and its inverse are given by
\begin{gather}
g^ix^j\mapsto(j)_{\xi^{-1}}!\sum_l\xi^{i(j+l)}\big(g^lx^j\big)^*
\end{gather}
and
\begin{gather}
(g^ix^j)^*\mapsto\dfrac{1}{p(j)_{\xi^{-1}}!}\sum_l\xi^{-l(i+j)}g^lx^j.
\end{gather}
\end{proof}

\begin{Corollary}\label{gens}The twisted double $\widehat{D}(T_p(\xi))$ is a quotient of $k\left\langle x,x',g,g'\right\rangle.$

\begin{itemize}\itemsep=0pt
\item The relations are
\begin{gather*}
x^p=x'^p=g^p-1=g'^p-1=0,
\\
gg'=g'g,\qquad
gx=\xi xg,\qquad
g'x'=\xi x'g',\qquad
gx'=\xi^{-1} x'g, \qquad
g'x=\xi^{-1} xg',
\\
xx'-\xi^{-1}x'x=1-\xi^{-1}g'^{-1}g.
\end{gather*}

\item The actions of $g'$ and $g$ on a $\widehat{D}(T_p(\xi))$-module $V$, yields a $(\mathbb{Z}/p)^2$-grading on $V$ by their eigenspaces, i.e., $g'$, $g$ act on $V_{ij}$ by $\xi^i$, $\xi^j$, respectively. Thus $x$ and $x'$ have degrees $(-1,1)$ and $(1,-1)$, respectively.

\item The $S^1$-action of $\sigma$ on $V_{ij}$ is
\begin{gather}\label{sigmaaction}
\sum_{l=0}^{p-1}\dfrac{\xi^{(i-l)(j+l)}}{(l)_{\xi^{-1}}!}x'^lx^l.
\end{gather}
\end{itemize}
\end{Corollary}

\begin{proof}
We use the identification of vector spaces $\widehat{D}(T_p(\xi))=(T_p(\xi))^*\ot T_p(\xi)$ as in Lemma~\ref{bas} followed by $(T_p(\xi))^*\ot T_p(\xi) \simeq T_p(\xi)\ot T_p(\xi)$ from Lemma~\ref{nenlem}. We~let
\begin{gather}
x'=x\ot 1, \qquad
x=1\ot x, \qquad
g'=g\ot 1, \qquad
g=1\ot g
\end{gather}
in the latter. To derive the rest of the relations we apply \eqref{mult1}. The action of $\sigma=\sum_{ij}\big(g^ix^j\big)^*\ot g^ix^j$ is computed on the graded components directly.
\end{proof}

Observe that it follows from Corollary~\ref{gens} that $gg'\in\widehat{D}(T_p(\xi))$ is central. Since we see that $(gg')^p=1$ so its action on $\widehat{D}(T_p(\xi))$ is diagonalizable with eigenvalues $\xi^s$, $s\in\zp$. Thus as an~algebra
\begin{gather}\label{split}
\widehat{D}(T_p(\xi))=\bigoplus_s \widehat{D}(T_p(\xi))/(gg'-\xi^s)
\end{gather}
so that it suffices to understand $\widehat{D}(T_p(\xi))/(gg'-\xi^s)$. There are two cases: $p=2$ and $p>2$. We~will begin by briefly discussing the latter (though without addressing the $S^1$-action), and then concentrate our attention on the former (with examining the $S^1$-action) to achieve the goal set out in the abstract.

\subsection[The case of p>2]{The case of $\boldsymbol{p>2}$}
Let $p>2$, then there exists a primitive $p$th root of unity $q\in k$ such that
\begin{gather}
q^2=\xi^{-1}.
\end{gather}
We then have
\begin{gather}\label{sbys}
\widehat{D}(T_p(\xi))/(gg'-\xi^s)\simeq u_q({\mathfrak{sl}}_2).
\end{gather}
More precisely, let
\begin{gather}
E=\frac{q^{s+1}}{q-q^{-1}}x', \qquad
F=xg',\qquad\text{and}\qquad
K=q^{s+1}g,
\end{gather}
so that $\widehat{D}(T_p(\xi))/(gg'-\xi^s)$ is generated by $E$, $F$, $K$ subject to
\begin{gather*}
E^p=F^p=K^p-1=0,
\\
[E,F]=\frac{K-K^{-1}}{q-q^{-1}},\qquad
KEK^{-1}=q^2E,\qquad \text{and}\qquad KFK^{-1}=q^{-2}F.
\end{gather*}
This shows that:
\begin{gather}\label{justasindh}
\widehat{D}(T_p(\xi))\simeq u_q({\mathfrak{sl}}_2)\ot\mathcal{O}_{\zp}\simeq u_q({\mathfrak{sl}}_2)\ot k\zp.
\end{gather}
As we will see below the case of $p=2$ is very different, in particular as $s$ varies, the algebra will change significantly whereas here it does not \eqref{sbys}.

\begin{Remark}\label{DD}
See~\cite{uqsl2}, where the Taft algebra is called the quantum ${\mathfrak{sl}}_2$ Borel algebra, and its Drinfeld double is computed. The result obtained is identical to ours in~\eqref{justasindh}, though we compute the twisted double. This is not surprising as we see from Appendix~\ref{appx} that our analysis of $\widehat{D}(H)$ can be interpreted as that of~$D(H)$, with $\sigma$ being a new ingredient. Note that more generally, a comparison between Drinfeld doubles of Nichols algebras and quantized universal enveloping algebras can be found in~\cite{extra}.
\end{Remark}

\subsection[The case of p=2]{The case of $\boldsymbol{p=2}$}
We need to describe the algebra $\widehat{D}(T_2(-1))$ in greater detail, paying particular attention to the element $\sigma$.

By \eqref{split} the category of DG-modules over $\widehat{D}(T_2(-1))[\theta]$ is a product of categories, $\Cc_0\times \Cc_1$, corresponding to the cases $s=0$ and $s=1$. We~will deal with both separately. More precisely, for a $\widehat{D}(T_2(-1))[\theta]$-module $V$, we decompose
\begin{gather}
V=(V_{00}\oplus V_{11})\oplus(V_{01}\oplus V_{10}).
\end{gather}

Observe that by the Corollary~\ref{gens} we have that
\begin{gather}
xx'+x'x=1+(-1)^s,
\end{gather}
so that
\begin{gather}\label{dd}
(x'x)^2=(1+(-1)^s)x'x.
\end{gather}
We see from \eqref{dd} that the minimal polynomial of $x'x$ depends only on $s$; this is exclusive to $p=2$ and makes this case tractable. Note that by \eqref{sigmaaction}:
\begin{gather}\label{sig}
\sigma|_{V_{00}}=1-x'x,\qquad \sigma|_{V_{11}}=-1+x'x,\qquad\text{and}\qquad \sigma|_{V_{01}}=\sigma|_{V_{10}}=1+x'x.
\end{gather} Let
\begin{gather}
D_s=\widehat{D}(T_2(-1))/(gg'-(-1)^s).
\end{gather}

We begin with $s=0$: The category of $D_0$-modules consists of $\mathbb{Z}/2$-graded vector spaces ($V=V_{00}\oplus V_{11}$) equipped with degree changing operators $x$ and $x'$ subject to the relation $xx'+x'x=2$. The action of $\sigma-1$ on $V_{00}$ is $-x'x$ and on $V_{11}$ it is $x'x-2$.

\begin{Lemma}
The category $\Cc_0$ consists of the mixed complexes of~{\rm \cite{kassel}}.
\end{Lemma}

 \begin{proof}
Note that $D_0/(\sigma-1)$-modules are just vector spaces. Indeed, let $y=x'/2$. For improved clarity, denote by $x_i$ and $y_i$ the action of $x$ and $y$, respectively, that originates at $V_{ii}$. After modding out by $\sigma-1$ we have by \eqref{sig} that $y_1x_0=0\implies x_1y_0=1$ and $y_0x_1=1$. So
\begin{gather}
y_0:V_{00}\simeq V_{11}: x_1.
\end{gather}
Furthermore, $x^2=y^2=0$ so that both $x_0$ and $y_1$ are the $0$ maps. Thus
\begin{gather}
V=V_{00}\oplus V_{11}\mapsto V_{00}
\end{gather}
establishes an equivalence of categories between $D_0/(\sigma-1)$-modules and $\vect$.

The action of $\sigma-1$ on $D_0$ is diagonalizable by \eqref{dd} and so by the proof of Lemma~\ref{diag} the algebras $D_0[\theta]$ ($d\theta=\sigma-1$) and $D_0/(\sigma-1)[\theta]$ ($d\theta=0$) are quasi-isomorphic. Thus, $\Cc_0$, being DG-equivalent to $D_0/(\sigma-1)[\theta]$-modules, is by the above discussion, equivalent to $k[\theta]$-modules. These are just the mixed complexes of~\cite{kassel}.
\end{proof}

 \begin{Remark}
 Note that not only does $\Cc_0$ consist of the usual mixed complexes but it also does not provide any evidence of the need for the mixed $aYD$-contramodules (see the proof above).
 \end{Remark}

Moving on to $s=1$ we find that things change for the better. Recall that $D_1$ is generated by
\begin{gather}
x,\ x',\ g
\end{gather}
subject to
\begin{gather}
x^2=x'^2=g^2-1=0, \qquad
xx'=-x'x, \qquad
gx=-xg, \qquad
gx'=-x'g.
\end{gather}
Furthermore, $\sigma-1$ acts as $x'x$ by \eqref{sig}.

\begin{Proposition}\label{notsscase}
Let $H=T_2(-1)$, then the mixed complexes in the category of stable anti-Yetter--Drinfeld contramodules are not DG-equivalent to the category of mixed anti-Yetter--Drinfeld contramodules.
\end{Proposition}
\begin{proof}
By the preceding discussion, i.e., the decomposition of the category into a product, it~suf\-fices to show that the categories $D_1[\theta]\text{-mod}$ (where $d\theta=x'x$) and $D_1/(x'x)[\theta]\text{-mod}$ (where $d\theta =0$) are not DG-equivalent.

Recall that the Hochschild cohomology $HH^i(C^\bullet)$ of a DG algebra $C^\bullet$ is an invariant of its DG category of modules~\cite{dginvar}. We~will compute $HH^{-1}$. In our simple case of a DG algebra $C=C^{-1}\stackrel{d}{\to} C^0$ concentrated in two degrees, we have
\begin{gather}
HH^{-1}(C)={\rm ker} \big(C^{-1}\stackrel{\alpha}{\to} C^0\oplus {\rm Hom}\big(C^0, C^{-1}\big)\big),
\end{gather}
where $\alpha(x)=(dx, [x,-])$.

The key observation here is that the center of $D_1$ is spanned by $1$, $xx'$, $xx'g$, while that of~$D_1/(x'x)$ is spanned by $1$. Thus, in the first case we get that $HH^{-1}$ is spanned by $xx'$ and~$xx'g$. In the second case it is spanned by $1$.
\end{proof}

\appendix

\section{Appendix}\label{appx}

Our purpose in this section is to compare $\widehat{D}(H)$ of Definition~\ref{dhhat} to the more familiar Drinfeld double $D(H)$ in the case of a finite dimensional Hopf algebra $H$. Though not original, see~\cite{quantumgroups} for Definition~\ref{yd} and~\cite{HKRS2} for Definition~\ref{ayd}, since our conventions differ from the usual ones we spell out the definitions again below (for $H$ finite dimensional):

\begin{Definition}\label{yd}
The algebra $D(H)$ is generated by $H$ and $H^*$ subject to the relations:
\begin{gather}
\chi h=h^2\chi\big(h^3(-)S^{-1}\big(h^1\big)\big),
\end{gather}
for $\chi\in H^*$ and $h\in H$. Thus $D(H)=H\ot H^*$ as vector spaces.
\end{Definition}

\begin{Definition}\label{ayd}
The algebra $D_a(H)$ is generated by $H$ and $H^*$ subject to the relations:
\begin{gather}
\chi h=h^2\chi\big(h^3(-)S\big(h^1\big)\big),
\end{gather}
for $\chi\in H^*$ and $h\in H$. Thus $D_a(H)=H\ot H^*$ as vector spaces. The central element is $\sigma=e_i\ot e^i$, where $e_i$ is any basis of $H$ and $e^i$ is its dual basis of $H^*$.
\end{Definition}

Note that modules over $D_a(H)$ as specified in Definition~\ref{ayd} can be identified with what is usually called left-right anti-Yetter--Drinfeld modules, i.e., left modules and right como\-du\-les~\cite{HKRS2}. Recall that if $H$ is finite dimensional then we have an $S^1$-equivariant equivalence between $\widehat{D}(H)$-modules and $D_a(H)$-modules~\cite{s2}. We~will thus focus on the comparison between $D_a(H)$ and~$D(H)$. It is known that in general they give very different categories of~modules~\cite{gentaft}. It is immediate that if $S^2={\rm Id}$ then the algebras in fact coincide, since the only difference bet\-ween them is $S$ in Definition~\ref{ayd} and $S^{-1}$ in Definition~\ref{yd}. Below we extend that observation slightly so as to cover our case of Taft algebras where we do not have $S^2={\rm Id}$, but instead we~get
\begin{gather}
S^2(h)=uhu^{-1}
\end{gather}
for some group-like element $u\in H$, i.e., $\Delta(u)=u\ot u$. For $T_p(\xi)$ we have $S^2(a)=g^{-1}ag$ with $\Delta g^{-1}=g^{-1} \ot g^{-1}$, so that $u=g^{-1}$.

Recall that Hopf algebras possessing such an element $u$ as above are called pivotal, see~\cite{pivot, spherical} for example. Their categories of representations are thus pivotal as well, i.e., equipped with a monoidal isomorphism from the identity functor to the double dual functor. This natural transformation is given by the action of $u\in H$; it is monoidal since $u$ is group-like and mapping to the double dual since $S^2(h)=uhu^{-1}$.

The following lemma is a straightforward computation but can be obtained from~\cite{gentaft}:

\begin{Lemma}\label{uhu}
Let $H$ be a finite dimensional Hopf algebra and suppose that there exists a $u\in H$ with $\Delta(u)=u\ot u$ such that $S^2(h)=uhu^{-1}$ for all $h\in H$. Then
\begin{align*}
D(H)&\to D_a(H),
\\
h\ot \chi&\mapsto h\ot \chi((-)u)
\end{align*} is an isomorphism of algebras.
\end{Lemma}

Thus for Taft algebras, Drinfeld doubles can play the role of $\widehat{D}(H)$, as long as we are careful to remember about the crucial central element $\sigma$.

\subsection*{Acknowledgements}

This research was supported in part by the NSERC Discovery Grant number 406709. The author wishes to thank the referees for their many helpful suggestions.

\pdfbookmark[1]{References}{ref}
\LastPageEnding


\begin{thebibliography}{99}
\footnotesize\itemsep=0pt

\bibitem{pivot}
Andruskiewitsch N., Angiono I., Garc\'{\i}a~Iglesias A., Torrecillas B., Vay
 C., From {H}opf algebras to tensor ca\-te\-gories, in Conformal Field Theories
 and Tensor Categories, \textit{Math. Lect. Peking Univ.}, \href{https://doi.org/10.1007/978-3-642-39383-9_1}{Springer}, Heidelberg, 2014,
 1--31, \href{https://arxiv.org/abs/1204.5807}{arXiv:1204.5807}.

\bibitem{extra}
Andruskiewitsch N., Radford D., Schneider H.J., Complete reducibility theorems
 for modules over pointed {H}opf algebras, \href{https://doi.org/10.1016/j.jalgebra.2010.06.002}{\textit{J.~Algebra}} \textbf{324}
 (2010), 2932--2970, \href{https://arxiv.org/abs/1001.3977}{arXiv:1001.3977}.

\bibitem{spherical}
Barrett J.W., Westbury B.W., Spherical categories, \href{https://doi.org/10.1006/aima.1998.1800}{\textit{Adv. Math.}}
 \textbf{143} (1999), 357--375, \href{https://arxiv.org/abs/hep-th/9310164}{arXiv:hep-th/9310164}.

\bibitem{it}
Ben-Zvi D., Francis J., Nadler D., Integral transforms and {D}rinfeld centers
 in derived algebraic geometry, \href{https://doi.org/10.1090/S0894-0347-10-00669-7}{\textit{J.~Amer. Math. Soc.}} \textbf{23}
 (2010), 909--966, \href{https://arxiv.org/abs/0805.0157}{arXiv:0805.0157}.

\bibitem{contra}
Brzezi\'nski T., Hopf-cyclic homology with contramodule coefficients, in
 Quantum groups and noncommutative spaces, \textit{Aspects Math.}, Vol.~E41,
 \href{https://doi.org/10.1007/978-3-8348-9831-9_1}{Vieweg + Teubner}, Wiesbaden, 2011, 1--8, \href{https://arxiv.org/abs/0806.0389}{arXiv:0806.0389}.

\bibitem{conmos}
Connes A., Moscovici H., Cyclic cohomology and {H}opf algebras, \href{https://doi.org/10.1023/A:1007527510226}{\textit{Lett.
 Math. Phys.}} \textbf{48} (1999), 97--108, \href{https://arxiv.org/abs/math.QA/9904154}{arXiv:math.QA/9904154}.

\bibitem{quantumgroups}
Drinfeld V.G., Quantum groups, in Proceedings of the {I}nternational {C}ongress
 of {M}athematicians, {V}ols.~1,~2 ({B}erkeley, {C}alif., 1986), Amer. Math.
 Soc., Providence, RI, 1987, 798--820.

\bibitem{HKRS1}
Hajac P.M., Khalkhali M., Rangipour B., Sommerh\"auser Y., Hopf-cyclic homology
 and cohomology with coefficients, \href{https://doi.org/10.1016/j.crma.2003.11.036}{\textit{C.~R.~Math. Acad. Sci. Paris}}
 \textbf{338} (2004), 667--672, \href{https://arxiv.org/abs/math.KT/0306288}{arXiv:math.KT/0306288}.

\bibitem{HKRS2}
Hajac P.M., Khalkhali M., Rangipour B., Sommerh\"auser Y., Stable
 anti-{Y}etter--{D}rinfeld modules, \href{https://doi.org/10.1016/j.crma.2003.11.037}{\textit{C.~R.~Math. Acad. Sci. Paris}}
 \textbf{338} (2004), 587--590, \href{https://arxiv.org/abs/math.QA/0405005}{arXiv:math.QA/0405005}.

\bibitem{gentaft}
Halbig S., Generalised {T}aft algebras and pairs in involution,
 \href{https://arxiv.org/abs/1908.10750}{arXiv:1908.10750}.

\bibitem{JS}
Jara P., \c{S}tefan D., Hopf-cyclic homology and relative cyclic homology of
 {H}opf--{G}alois extensions, \href{https://doi.org/10.1017/S0024611506015772}{\textit{Proc. London Math. Soc.}} \textbf{93}
 (2006), 138--174.

\bibitem{kassel}
Kassel C., Cyclic homology, comodules, and mixed complexes, \href{https://doi.org/10.1016/0021-8693(87)90086-X}{\textit{J.~Algebra}}
 \textbf{107} (1987), 195--216.

\bibitem{uqsl2}
Kerler T., Mapping class group actions on quantum doubles, \href{https://doi.org/10.1007/BF02101554}{\textit{Comm. Math.
 Phys.}} \textbf{168} (1995), 353--388, \href{https://arxiv.org/abs/hep-th/9402017}{arXiv:hep-th/9402017}.

\bibitem{sqid2}
Larson R.G., Radford D.E., Finite-dimensional cosemisimple {H}opf algebras in
 characteristic {$0$} are semisimple, \href{https://doi.org/10.1016/0021-8693(88)90107-X}{\textit{J.~Algebra}} \textbf{117} (1988),
 267--289.

\bibitem{sqid}
Larson R.G., Radford D.E., Semisimple cosemisimple {H}opf algebras,
 \href{https://doi.org/10.2307/2374545}{\textit{Amer.~J. Math.}} \textbf{110} (1988), 187--195.

\bibitem{nen}
Nenciu A., Quasitriangular pointed {H}opf algebras constructed by {O}re
 extensions, \href{https://doi.org/10.1023/B:ALGE.0000026785.03997.60}{\textit{Algebr. Represent. Theory}} \textbf{7} (2004), 159--172.

\bibitem{dhss}
Radford D.E., Minimal quasitriangular {H}opf algebras, \href{https://doi.org/10.1006/jabr.1993.1102}{\textit{J.~Algebra}}
 \textbf{157} (1993), 285--315.

\bibitem{selfdual}
Radford D.E., Westreich S., Trace-like functionals on the double of the {T}aft
 {H}opf algebra, \href{https://doi.org/10.1016/j.jalgebra.2004.04.023}{\textit{J.~Algebra}} \textbf{301} (2006), 1--34.

\bibitem{s2}
Shapiro I., On the anti-{Y}etter--{D}rinfeld module-contramodule
 correspondence, \href{https://doi.org/10.4171/JNCG/327}{\textit{J.~Noncommut. Geom.}} \textbf{13} (2019), 473--497,
 \href{https://arxiv.org/abs/1704.06552}{arXiv:1704.06552}.

\bibitem{chern}
Shapiro I., Categorified {C}hern character and cyclic cohomology,
 \href{https://arxiv.org/abs/1904.04230}{arXiv:1904.04230}.

\bibitem{taft}
Taft E.J., The order of the antipode of finite-dimensional {H}opf algebra,
 \href{https://doi.org/10.1073/pnas.68.11.2631}{\textit{Proc. Nat. Acad. Sci. USA}} \textbf{68} (1971), 2631--2633.

\bibitem{dginvar}
To\"en B., The homotopy theory of {$dg$}-categories and derived {M}orita
 theory, \href{https://doi.org/10.1007/s00222-006-0025-y}{\textit{Invent. Math.}} \textbf{167} (2007), 615--667,
 \href{https://arxiv.org/abs/math.AG/0408337}{arXiv:math.AG/0408337}.

\end{thebibliography}
\end{document}